\documentclass[a4paper,12pt]{article}
\usepackage[english]{babel}
\usepackage[utf8]{inputenc}
\usepackage{amssymb}
\usepackage{amsmath}

\usepackage{graphicx}
\usepackage{epstopdf}

\usepackage{verbatim}

\def\cF{\mathcal{F}}
\def\cG{\mathcal{G}}

\def\cP{\mathcal{P}}
\def\cQ{\mathcal{Q}}

\def\cT{\mathcal{T}}

\def\C{{\mathbb C}}

\def\E{{\mathbb E}}

\def\N{{\mathbb N}}

\def\P{{\mathbb P}}
\def\R{{\mathbb R}}

\def\Z{{\mathbb Z}}

\def\pv{\overrightarrow{p}}
\def\qv{\overrightarrow{q}}

\def\qvc{(\overrightarrow{q})'}

\def\Lqvc{L_{(\overrightarrow{q})'}}

\def\n{\nonumber}

\def\vpc{(\overrightarrow{p})'}

\author{Kristóf Szarvas \thanks{This research was supported by the Hungarian National Research, Development and Innovation Office - NKFIH, K115804 and KH130426.}
\\ Department of Numerical Analysis\\
E\"otv\"os L. University \\
P\'azm\'any P. s\'et\'any 1/C., H-1117 Budapest, Hungary \\
e-mail: kristof@inf.elte.hu\\
\and Ferenc Weisz
\\ Department of Numerical Analysis\\
E\"otv\"os L. University \\
P\'azm\'any P. s\'et\'any 1/C., H-1117 Budapest, Hungary \\
e-mail: weisz@inf.elte.hu}

\title{Mixed martingale Hardy spaces
}

\date{}

\newtheorem{thm}{Theorem}
\newtheorem{cor}{Corollary}
\newtheorem{lem}{Lemma}
\newtheorem{remark}{Remark}

\newenvironment{proof}{\begin{trivlist} \item[] \textbf{Proof.}}{\quad \rule{2mm}{2mm} \end{trivlist}}

\begin{document}

\maketitle

\begin{abstract}
In this paper we consider the martingale Hardy spaces defined with the help of the mixed $L_{\pv}$-norm. Five mixed martingale Hardy spaces will be investigated: $H_{\pv}^{s}$, $H_{\pv}^S$, $H_{\pv}^M$, $\cP_{\pv}$ and $\cQ_{\pv}$. Several results are proved for these spaces, like atomic decompositions, Doob's inequality, boundedness, martingale inequalities and the generalization of the well-known Burkholder-Davis-Gundy inequality.
\end{abstract}




\section{Introduction}

The mixed Lebesgue spaces were introduced in 1961 by Benedek and Panzone \cite{Benedek1961}. They considered the Descartes product $(\Omega,\cF,\P)$ of the probability spaces $(\Omega_{i},\cF^{i},\P_{i})$,
where $\Omega = \prod_{i=1}^{d} \Omega_{i}$, $\cF$ is generated by $\prod_{i=1}^{d} \cF^{i}$ and $\P$ is generated by $\prod_{i=1}^{d} \P_{i}$. The mixed $L_{\pv}$-norm of the measurable function $f$ is defined as a number obtained after taking successively the $L_{p_{1}}$-norm of $f$ in the variable $x_{1}$, the $L_{p_{2}}$-norm of $f$ in the variable $x_{2}$, $\ldots$, the $L_{p_{d}}$-norm of $f$ in the variable $x_{d}$. Some basic properties of the spaces $L_{\pv}$ were proved in \cite{Benedek1961}, such as the well known Hölder's inequality or the duality theorem for $L_{\pv}$-norm (see Lemma \ref{norma-sup}). The boundedness of operators on mixed-norm spaces have been studied for instance by Fernandez \cite{LassFernandez1987} and Stefanov and Torres \cite{Stefanov2004}. Using the mixed Lebesgue spaces, Lizorkin \cite{Lizorkin1971} considered Fourier integrals and estimations for convolutions. Torres and Ward \cite{Torres2015} gave the wavelet characterization of the space $L_{\pv}(\R^n)$. For more about mixed-norm spaces the papers \cite{Antonic2016,Chen,Cleanthous2017a,Hart2018,Ho2016,Ho2018,Huang2019b,Igari1986,Nogayama2019,Stefanov2004} are referred.

Since 1970, the theory of Hardy spaces has been developed very quickly (see e.g. Fefferman and Stein \cite{fest}, Stein \cite{st1}, Grafakos \cite{gra}). Parallel, a similar theory was evolved for martingale Hardy spaces (see e.g. Garsia \cite{ga}, Long \cite{lo} and Weisz \cite{wk}). Recently several papers were published about the generalization of Hardy spaces. For example, (anisotropic) Hardy spaces with variable exponents were considered in Nakai and Sawano \cite{Nakai2012}, Yan, et al. \cite{Yan2016}, Jiao et al. \cite{Jiao2019}, Liu et al. \cite{Liu2018a} and \cite{Liu2019}. Moreover Musielak-Orlicz-Hardy spaces were studied in Yang et al. \cite{Yang2017}. These results were also investigated for martingale Hardy spaces in Jiao et al. \cite{Jiao2016} and \cite{w-mart-var} and Xie et al. \cite{Xie2019}. The mixed norm classical Hardy spaces were intensively studied by Huang et al. in \cite{Huang2019}, \cite{Huang2019a} and \cite{Huang2019b}. In this paper we will develop a similar theory for mixed norm martingale Hardy spaces.

The classical martingale Hardy spaces ($H_{p}^{s}$, $H_{p}^{S}$, $H_{p}^{\ast}$, $\cP_{p}$, $\cQ_{p}$) have a long history and the results of this topic can be well applied in the Fourier analysis. In the celebrated work of Burkholder and Gundy \cite{bugu} it was proved that the $L_{p}$ norms of the maximal function and the quadratic variation, that is the spaces $H_{p}^M$ and $H_{p}^S$, are equivalent for $1 < p < \infty$. In the same year, Davis \cite{da1} extended this result for $p=1$. In the classical case, Herz \cite{Herz1974} and Weisz \cite{wk} gave one of the most powerful techniques in the theory of martingale Hardy spaces, the so-called atomic decomposition. Some boundedness results, duality theorems, martingale inequalities and interpolation results can be proved with the help of atomic decomposition. Details for the martingale Hardy spaces can be found in Burkholder and Gundy \cite{Burkholder1966}, \cite{bu1}, \cite{bugu}, Garsia \cite{ga}, Long \cite{lo} or Weisz \cite{wdual1,wk}. For the application of martingale Hardy spaces in Fourier analysis see Gát \cite{Gat2018,Gat2019}, Goginava \cite{gog3,gog6} or Weisz \cite{wk,wk2}.

In this paper we will introduce five mixed martingale Hardy spaces: $H_{\pv}^{s}$, $H_{\pv}^{S}$, $H_{\pv}^{M}$, $\cP_{\pv}$ and $\cQ_{\pv}$. In Section \ref{doob's inequality}, Doob's inequality will be proved, that is, we will show that
$$
\left\| \sup_{n \in \N} \E_{n}f \right\|_{\pv} \leq c \, \|f\|_{\pv}
$$
for all $f \in L_{\pv}$, where $1 < \pv < \infty$. In Section \ref{atomic decomp}, we give the atomic decomposition for the five mixed martingale Hardy spaces. Using the atomic decompositions and Doob's inequality, the boundedness of general $\sigma$-subadditive operators from $H_{\pv}^{s}$ to $L_{\pv}$, from $\cP_{\pv}$ to $L_{\pv}$ and from $\cQ_{\pv}$ to $L_{\pv}$ can be proved (see Theorems \ref{T korl Hps-bol} and \ref{T korl HpS-bol es HpM-bol}). With the help of these general boundedness theorems, several martingale inequalities will be proved in Section \ref{mart ineq} (see Corollary \ref{mart-ineq-7}). We will show, that if the stochastic basis $(\cF_{n})$ is regular, then the five martingale Hardy spaces are equivalent. As a consequence of Doob's inequality, the well-known Burkholder-Davis-Gundy inequality can be shown. Moreover, if the stochastic basis is regular, then the so-called martingale transform is bounded on $L_{\pv}$.

We denote by $C$ a positive constant, which can vary from line to
line, and denote by $C_p$ a constant depending only on $p$. The symbol $A \sim B$
means that there exist constants $\alpha, \beta > 0$ such that $\alpha A \leq B \leq \beta A$.

\section{Backgrounds}

\subsection{Mixed Lebesgue spaces}

We will start with the definition of the mixed Lebesgue spaces. To this, let $1 \leq d \in \N$ and $(\Omega_{i}, \mathcal{F}^{i}, \P_{i})$ be probability spaces for $i=1,\dots,d$, and $\pv := \left(p_{1},\ldots,p_{d}\right)$ with $0 < p_{i} \leq \infty$. Consider the product space $(\Omega,\cF,\P)$, where $\Omega = \prod_{i=1}^{d} \Omega_{i}$, $\cF$ is generated by $\prod_{i=1}^{d} \cF^{i}$ and $\P$ is generated by $\prod_{i=1}^{d} \P_{i}$. A measurable function $f : \Omega \to \R$ belongs to the mixed $L_{\pv}$ space if
\begin{eqnarray*}
\left\|f\right\|_{\pv} := \left\|f\right\|_{(p_{1},\ldots,p_{d})} &:=& \left\| \ldots \left\|f\right\|_{L_{p_{1}}(dx_{1})} \ldots \right\|_{L_{p_{d}}(dx_{d})} \\
&=& \left( \int_{\Omega_{d}} \ldots \left( \int_{\Omega_{1}} \left| f(x_{1},\ldots,x_{d})\right|^{p_{1}} \, dx_{1} \right)^{p_{2}/p_{1}} \ldots dx_{d} \right)^{1/p_{d}} < \infty
\end{eqnarray*}
with the usual modification if $p_{j} = \infty$ for some $j \in \left\{1,\ldots,d\right\}$. If for some $0 < p \leq \infty$, $\pv = \left(p, \ldots, p\right)$, then we get back the classical Lebesgue space, that is, in this case, $L_{\pv} = L_{p}$. Throughout the paper, $0<\pv \leq \infty$ will mean that the coordinates of $\pv$ satisfy the previous condition, e.g., for all $i=1,\ldots,d$, $0 < p_{i} \leq \infty$. The conjugate exponent vector of $\pv$ will be denoted by $\vpc$, that is, if $\vpc = \left(p_{1}', \ldots,p_{d}'\right)$, then $1/p_{i} + 1/p_{i}' = 1$ $(i=1,\ldots,d)$. For $\alpha > 0$, $\pv / \alpha := \left(p_{1}/\alpha, \ldots, p_{d}/\alpha\right)$. Benedek and Panzone \cite{Benedek1961} proved some basic properties for the mixed Lebesgue space.

\begin{lem}\label{holder}
If $1 \leq \pv \leq \infty$, then for all $f \in L_{\pv}$ and $g \in L_{\vpc}$, $fg \in L_{1}$ and
$$
\left\|fg\right\|_{1} = \int_{\Omega_{d}} \ldots \int_{\Omega_{1}} \left| f(x) g(x) \right| \, d\P(x) \leq \left\|f\right\|_{\pv} \left\|g\right\|_{\vpc}.
$$
\end{lem}

\begin{lem}\label{norma-sup}
If $1 \leq \pv \leq \infty$ and $f \in L_{\pv}$, then
$$
\left\| f \right\|_{\pv} = \sup_{\|g\|_{\vpc} \leq 1} \left| \int_{\Omega} f g \, d\P \right|.
$$
\end{lem}

\subsection{Martingale Hardy spaces}

Suppose that the $\sigma$-algebra $\cF_{n}^{i} \subset \cF^{i}$ $(n \in \N, \, i=1,\ldots,d)$, $(\cF_{n}^{i})_{n \in \N}$ is increasing and $\cF^{i} = \sigma\left( \cup_{n \in \N} \cF_{n}^{i} \right)$. Let $\cF_{n} = \sigma\left( \prod_{i=1}^{d} \cF_{n}^{i} \right)$. The expectation and  conditional expectation operators relative to $\mathcal{F}_{n}$ are denoted by $\mathbb{E}$ and $\mathbb{E}_{n}$, respectively. An integrable sequence $f = \left(f_{n}\right)_{n \in \N}$ is said to be a \emph{martingale} if
\begin{enumerate}
\item $(f_{n})_{n \in \N}$ is \emph{adapted}, that is for all $n \in \N$, $f_{n}$ is $\mathcal{F}_{n}$-measurable;
\item $\mathbb{E}_{n} f_{m} = f_{n}$ in case $n \leq m$.
\end{enumerate}
For $n \in \N$, the \emph{martingale difference} is defined by $d_{n}f := f_{n}-f_{n-1}$, where $f= \left(f_{n}\right)_{n \in \N}$ is a martingale and $f_{0}:=f_{-1}:=0$. Thus $d_{0}=0$. If for all $n \in \N$, $f_{n} \in L_{\pv}$, then $f$ is called an $L_{\pv}$-martingale. Moreover, if
$$
\left\|f\right\|_{\pv} := \sup_{n \in \N} \left\|f_{n}\right\|_{\pv} < \infty,
$$
then $f$ is called an $L_{\pv}$-bounded martingale and it will be denoted by $f \in L_{\pv}$. The map $\nu: \Omega \to \N \cup \{\infty\}$ is called a \emph{stopping time relative to $(\cF_{n})$} if for all $n \in \N$, $\{ \nu = n \} \in \cF_{n}$.

For a martingale $f=(f_{n})$ and a stopping time $\nu$, the \emph{stopped martingale} is defined by
$$
f_{n}^{\nu} = \sum_{m=0}^{n} d_{m}f \, \chi_{\{ \nu \geq m \}}.
$$

Let us define the \emph{maximal function}, the \emph{quadratic variation} and the \emph{conditional quadratic variation} of the martingale $f$ relative to $(\Omega, \mathcal{F}, \mathbb{P}, (\mathcal{F}_{n})_{n \in \N})$ by
\begin{eqnarray*}
M_{m} \left(f\right) &:=& \sup_{n \leq m } \left| f_{n} \right|, \quad M \left(f\right) := \sup_{n \in \N} \left| f_{n}\right|\\
S_{m} \left(f\right) &:=& \left( \sum_{n=0}^m \left| d_{n}f\right|^2 \right)^{1/2}, \quad S \left(f\right) := \left( \sum_{n = 0}^{\infty} \left| d_{n}f\right|^2 \right)^{1/2}\\
s_{m} \left(f\right) &:=& \left( \sum_{n=0}^{m} \mathbb{E}_{n-1} \left| d_{n}f\right|^2 \right)^{1/2}, \quad s \left(f\right) := \left( \sum_{n=0}^{\infty} \mathbb{E}_{n-1} \left| d_{n}f\right|^2 \right)^{1/2}.
\end{eqnarray*}
The set of the sequences $(\lambda_{n})_{n \in \N}$ of non-decreasing, non-negative and adapted functions with $\lambda_{\infty} := \lim_{n \to \infty} \lambda_{n}$ is denoted by $\Lambda$. With the help of the previous operators, the mixed martingale Hardy spaces can be defined as follows:
\begin{eqnarray*}
H_{\pv}^{M} &:=& \left\{ f = \left(f_{n}\right)_{n \in \N} : \left\|f\right\|_{H_{\pv}^{M}} := \left\| M(f) \right\|_{\pv} < \infty \right\};\\
H_{\pv}^{S} &:=& \left\{ f = \left(f_{n}\right)_{n \in \N} : \left\|f\right\|_{H_{\pv}^{S}} := \left\| S(f) \right\|_{\pv} < \infty \right\};\\
H_{\pv}^{s} &:=& \left\{ f = \left(f_{n}\right)_{n \in \N} : \left\|f\right\|_{H_{\pv}^{s}} := \left\| s(f) \right\|_{\pv} < \infty \right\};\\
\mathcal{Q}_{\pv} &:=& \left\{ f = \left(f_{n}\right)_{n \in \N} : \exists \left(\lambda_{n}\right)_{n \in \N} \in \Lambda, \, \mbox{such that} \, S_{n} \left(f\right) \leq \lambda_{n-1}, \, \lambda_{\infty} \in L_{\pv} \right\},\\
& & \left\| f \right\|_{\mathcal{Q}_{\pv}} := \inf_{(\lambda_{n}) \in \Lambda} \left\| \lambda_{\infty} \right\|_{\pv};\\
\mathcal{P}_{\pv} &:=& \left\{ f = \left(f_{n}\right)_{n \in \N} : \exists \left(\lambda_{n}\right)_{n \in \N} \in \Lambda, \, \mbox{such that} \, \left| f_{n} \right| \leq \lambda_{n-1}, \, \lambda_{\infty} \in L_{\pv} \right\},\\
& & \left\| f \right\|_{\mathcal{P}_{\pv}} := \inf_{(\lambda_{n}) \in \Lambda} \left\| \lambda_{\infty} \right\|_{\pv}.
\end{eqnarray*}

\section{Doob's inequality}\label{doob's inequality}

In this section, we will prove that the maximal operator $M$ is bounded on the space $L_{\overrightarrow{p}}$ for $1 < \pv < \infty$. Recall that $Mf = \sup_{n} \left| f_{n} \right|$. If $1 < \pv < \infty$ and the martingale $(f_{n})_{n \in \N} \in L_{\overrightarrow{p}}$, then there is a function $g \in L_{\overrightarrow{p}}$ such that for all $n \in \N$, $f_{n} = \mathbb{E}_{n}g$. For $k \in \{1,\ldots,d \}$ and $m \in \N$, let us denote
$$
\cF_{\infty,\ldots,\infty,m,\infty,\ldots,\infty} := \sigma\left( \cF^{1} \times \cdots \times \cF^{k-1} \times \cF_{m}^{k} \times \cF^{k+1} \times \cdots \times \cF^{d} \right),
$$
where $m$ stands in the $k$-th position. The conditional expectation operator relative to $\cF_{\infty,\ldots,\infty,m,\infty,\ldots,\infty}$ is denoted by $\E_{\infty,\ldots,\infty,m,\infty,\ldots,\infty}$. We need the following maximal operators: for an integrable function $f$, let
$$
M_{k}f := \sup_{m \in \N} \left| \mathbb{E}_{\infty,\ldots,\infty,m,\infty,\ldots,\infty} f \right|, \quad \widetilde{M}f := M_{d} \circ M_{d-1} \circ \ldots \circ M_{1} f.
$$
It is clear that
$$
Mf \leq \widetilde{M}f.
$$
For $0 < p < \infty$ and $w > 0$, the weighted space $L_{p}(w)$ consists of all functions $f$, for which
$$
\left\| f \right\|_{L_{p}(w)} := \left( \int_{\Omega} |f|^p \, w d\P \right)^{1/p} < \infty.
$$
We need the following lemma.

\begin{lem}\label{M atdobhato}
Let $\varphi$ be a positive function. Then for all $r > 1$, $M$ is bounded from $L_{r} \left(M \varphi\right)$ to $L_r \left(\varphi\right)$, that is for all $f \in L_{r}\left(M \varphi\right)$ we have
\begin{equation}\label{eq61}
\int_{\Omega} |M f|^r \, \varphi d\P \leq c_{r} \int_{\Omega} \left| f\right|^r \, M \varphi \, d\P.
\end{equation}
\end{lem}
\begin{proof}
It is easy to see, that the operator $M$ is bounded from $L_{\infty}(M\varphi)$ to $L_{\infty}(\varphi)$. We will prove that $M$ is bounded from $L_{1}(M\varphi)$ to $L_{1,\infty}(\varphi)$ as well, where $L_{1,\infty}(\varphi)$ denotes the weak-$L_{1}(\varphi)$ space. From this, it follows by interpolation (see e.g. \cite{belo}) that for all $r > 1$, the operator $M$ is bounded from $L_{r}(M\varphi)$ to $L_{r}(\varphi)$, in oder words, \eqref{eq61} holds. Let $\varrho > 0$ arbitrary and let $\nu_{\varrho} := \inf \left\{ n \in \N : \left| f_{n} \right| > \varrho \right\}$. Then using that $\left\{ \nu_\varrho < \infty \right\} = \left\{ Mf > \varrho \right\}$, we get that
\begin{eqnarray*}
\left\|Mf\right\|_{L_{1,\infty}(\varphi)} &:=& \varrho \int_{ \left\{ Mf > \varrho \right\}} \varphi \, d\P = \varrho \sum_{k \in \N} \int_{ \left\{ \nu_{\varrho} = k \right\}} \varphi \, d\P \leq \sum_{k \in \N} \int_{\left\{\nu_{\varrho}=k\right\}} |f_{k}| \, \varphi \, d\P\\
&=& \sum_{k \in \N} \int_{\left\{\nu_{\varrho}=k\right\}} \mathbb{E}_{k} |f| \, \varphi \, d\P = \sum_{k \in \N} \int_{\left\{\nu_{\varrho}=k\right\}} |f|  \, \mathbb{E}_{k}\varphi \, d\P \\
&\leq& \sum_{k \in \N} \int_{\left\{\nu_{\varrho}=k\right\}} |f|  \, M \varphi  \, d\P \leq \int_{\Omega} |f| \, M\varphi \, d\P = \left\|f\right\|_{L_{1}(M\varphi)},
\end{eqnarray*}
which finishes the proof.
\end{proof}
Now we prove that $M_{d}$ is bounded on $L_{\pv}$.

\begin{thm}\label{tbagby}
    For all $1 < \pv < \infty$, $M_{d}$ is bounded on $L_{\pv}$, that is for all $f \in L_{\pv}$,
    $$
    \|M_{d}f\|_{\pv} \leq c \|f\|_{\pv}.
    $$
\end{thm}
\begin{proof}
We will prove the theorem by induction in $d$. If $d=1$, then the function $f$ has only $1$ variable and the theorem holds (see, e.g., Weisz \cite{wk}). Suppose that the theorem is true for some fixed $d \in \N$ and for all $1 < \pv = (p_{1},\ldots,p_{d}) < \infty$ and $f \in L_{\pv}$. For a function $f$ with $d$ variables and for the vector $(p_{1},\ldots,p_{k})$, let us denote
\begin{eqnarray*}
T_{(p_{1},\ldots,p_{k})}f(x_{k+1},\ldots,x_{d}) &:=& \|f(\cdot,\ldots,\cdot,x_{k+1},\ldots,x_{d})\|_{(p_{1},\ldots,p_{k})}\\
&=& \left( \int_{\Omega_{k}} \cdots  \left( \int_{\Omega_{1}} |f(x_{1},\ldots,x_{d})|^{p_{1}} \, dx_{1} \right)^{p_{2}/p_{1}} \cdots  \, dx_{k} \right)^{1/p_{k}}.
\end{eqnarray*}
Using this notation, the condition of the induction can be written in the form
$$
\int_{\Omega_{d}} \left[T_{(p_{1},\ldots,p_{d-1})}(M_{d}f) \right]^{p_{d}}(x_{d}) \, dx_{d} \leq c \int_{\Omega_{d}} \left[T_{(p_{1},\ldots,p_{d-1})}(f) \right]^{p_{d}}(x_{d}) \, dx_{d}.
$$
We will show that if $1 < p_{d+1} < \infty$, then for all $f \in L_{(p_{1},\ldots, p_{d+1})}$,
$$
\int_{\Omega_{d+1}} \left[T_{(p_{1},\ldots,p_{d})}(M_{d+1}f) \right]^{p_{d+1}}(x_{d+1}) \, dx_{d+1} \leq c \int_{\Omega_{d+1}} \left[T_{(p_{1},\ldots,p_{d})}(f) \right]^{p_{d+1}}(x_{d+1}) \, dx_{d+1},
$$
where $f$ has $d+1$ variable and the maximal operator $M_{d+1}$ is taken in the variable $x_{d+1}$. If $p_{1} = \infty$, then
$$
|f(x_{1},\ldots,x_{d},x_{d+1})| \leq \sup_{x_{1} \in \Omega_{1}}|f(x_{1},\ldots,x_{d},x_{d+1})| = T_{\infty}f(x_{2},\ldots,x_{d},x_{d+1}).
$$
Hence $M_{d+1}f(x_{1},\ldots,x_{d},x_{d+1}) \leq M_{d+1}(T_{\infty}f)(x_{2},\ldots,x_{d},x_{d+1}),
$ and therefore
\begin{eqnarray*}
T_{\infty}(M_{d+1}f)(x_{2},\ldots,x_{d},x_{d+1}) &=& \sup_{x_{1} \in \Omega_{1}} M_{d+1}f(x_{1},\ldots,x_{d},x_{d+1})\\
&\leq& M_{d+1}(T_{\infty}f)(x_{2},\ldots,x_{d},x_{d+1}).
\end{eqnarray*}
So we get that
\begin{eqnarray}\label{indukcio elott}
\lefteqn{\int_{\Omega_{d+1}} \left[T_{(p_{2},\ldots,p_{d})}(T_{\infty}(M_{d+1}f)) \right]^{p_{d+1}}(x_{d+1}) \, dx_{d+1}} \n\\
&\leq& \int_{\Omega_{d+1}} \left[T_{(p_{2},\ldots,p_{d})}(M_{d+1}(T_{\infty}f)) \right]^{p_{d+1}}(x_{d+1}) \, dx_{d+1}.
\end{eqnarray}
Here the function $M_{d+1}(T_{\infty}f)$ has $d$ variables: $x_{2},\ldots,x_{d+1}$ and the maximal operator is taken over the $d$-th variable, that is over $x_{d+1}$. Therefore by induction we have that \eqref{indukcio elott} can be estimated by
$$
c \int_{\Omega_{d+1}} \left[T_{(p_{2},\ldots,p_{d})}(T_{\infty}f) \right]^{p_{d+1}}(x_{d+1}) \, dx_{d+1},
$$
which means that
\begin{equation}\label{infinityre ok}
\|M_{d+1}f\|_{(\infty,p_{2},\ldots,p_{d},p_{d+1})} \leq c \|f\|_{(\infty,p_{2},\ldots,p_{d},p_{d+1})},
\end{equation}
so the theorem holds for $p_{1}=\infty$. Now let choose a number $r$ for which $1 < r < \min\{ p_{2}, \ldots, p_{d}, p_{d+1} \}$. It will be shown that
$$
\|M_{d+1}f\|_{(r,p_{2},\ldots,p_{d},p_{d+1})} \leq c \|f\|_{(r,p_{2},\ldots,p_{d},p_{d+1})}.
$$
It is easy to see that
$$
\|M_{d+1}f\|_{(r,p_{2},\ldots,p_{d+1})} = \left\|[T_{r}(M_{d+1}f)]^r \right\|_{\left(\frac{p_{2}}{r},\ldots,\frac{p_{d+1}}{r} \right)}^{1/r}.
$$
Let $\qv := \left(\frac{p_{2}}{r},\ldots,\frac{p_{d+1}}{r} \right)$. Then the vector $\qv$ has $d$ coordinates and $1 < \qv < \infty$. Using Lemma \ref{norma-sup},
\begin{eqnarray*}
\lefteqn{\left\|[T_{r}(M_{d+1}f)]^r \right\|_{\qv}}\\
&=& \sup_{\substack{\varphi \in \Lqvc \\ \|\varphi\|_{\qvc} \leq 1}} \left| \int_{\Omega_{d+1}} \cdots \int_{\Omega_{2}} \left[T_{r}(M_{d+1}f) \right]^r(x_{2},\ldots,x_{d+1}) \, \varphi(x_{2},\ldots,x_{d+1})  \, dx_{2}\cdots dx_{d+1} \right|.
\end{eqnarray*}
We can suppose that $\varphi > 0$. Then
\begin{eqnarray}\label{atdobas elott}
\lefteqn{\int_{\Omega_{d+1}} \cdots \int_{\Omega_{2}} \left[T_{r}(M_{d+1}f) \right]^r(x_{2},\ldots,x_{d+1}) \, \varphi(x_{2},\ldots,x_{d+1})  \, dx_{2}\cdots dx_{d+1} }\\
&=& \int_{\Omega_{1}} \left( \int_{\Omega_{d}} \cdots \int_{\Omega_{2}} \left( \int_{\Omega_{d+1}} |M_{d+1} f|^r(x_{1},\ldots,x_{d+1}) \, \varphi(x_{2},\ldots,x_{d+1}) \, dx_{d+1} \right) dx_{2} \cdots dx_{d} \right) dx_{1} \n.
\end{eqnarray}
Since $1 < r < \infty$, applying Lemma \ref{M atdobhato} for the variable $x_{d+1}$, we have that for all fixed $x_{1},\ldots,x_{d}$,
\begin{eqnarray*}
\lefteqn{\int_{\Omega_{d+1}} |M_{d+1} f|^r(x_{1},\ldots,x_{d+1}) \, \varphi(x_{2},\ldots,x_{d+1}) \, dx_{d+1}  }\\
&\leq& c \int_{\Omega_{d+1}} |f(x_{1},\ldots,x_{d+1})|^r \, M_{d+1}\varphi(x_{2},\ldots,x_{d+1}) \, dx_{d+1}
\end{eqnarray*}
Hence \eqref{atdobas elott} can be estimated by
\begin{eqnarray}\label{eq3}
\lefteqn{c \int_{\Omega_{d+1}} \cdots \int_{\Omega_{2}} \left( \int_{\Omega_{1}} |f(x_{1},\ldots,x_{d+1})|^r \, dx_{1} \right) \, M_{d+1}\varphi(x_{2},\ldots,x_{d+1}) \, dx_{2} \cdots dx_{d+1}} \n\\
&=& c \int_{\Omega_{d+1}} \cdots \int_{\Omega_{2}} [T_{r}(f)]^r(x_{2},\ldots,x_{d+1}) \, M_{d+1}\varphi(x_{2},\ldots,x_{d+1}) \, dx_{2} \cdots dx_{d+1} \n\\
&\leq& c \left\| \left[T_{r}(f) \right]^r \right\|_{\qv} \, \left\| M_{d+1} \varphi \right\|_{\qvc}.
\end{eqnarray}
Here $M_{d+1}\varphi$ is a function with $d$ variables, the vector $\qvc$ has $d$ coordinates such that $1 < \qvc < \infty$ and the maximal operator is taken in the $d$-th coordinate, that is over the variable $x_{d+1}$. So, by induction we get that
$$
\left\| M_{d+1} \varphi \right\|_{\qvc} \leq c \left\| \varphi \right\|_{\qvc} \leq c.
$$
Therefore we can estimate \eqref{eq3} by
\begin{eqnarray*}
\lefteqn{ c \left\| \left[T_{r}(f) \right]^r \right\|_{\qv} } \\
&=&  c \left( \int_{\Omega_{d+1}} \cdots \left( \int_{\Omega_{2}} \left( \int_{\Omega_{1}} |f(x_{1},x_{2},\ldots,x_{d+1})|^r \, dx_{1} \right)^{p_{2}/r} \, dx_{2} \right)^{p_{3}/r \cdot r/p_{2}} \cdots dx_{d+1} \right)^{r/p_{d+1}} \\
&=& \|f\|_{(r,p_{2},\cdots,p_{d+1})}^{r}.
\end{eqnarray*}
Consequently,
\begin{equation}\label{r-re is ok}
\|M_{d+1}f\|_{(r,p_{2},\cdots,p_{d+1})} \leq c \|f\|_{(r,p_{1},\cdots,p_{d+1})}.
\end{equation}
Combining the results \eqref{infinityre ok} and \eqref{r-re is ok} we get by interpolation that for all $1 < p_{1} < \infty$
$$
\|M_{d+1}f\|_{(p_{1},p_{2},\cdots,p_{d+1})} \leq c \|f\|_{(p_{1},p_{1},\cdots,p_{d+1})}.
$$
Using induction, the proof is complete.
\end{proof}

\begin{remark}\label{remark1}
Using the proof of the previous theorem, this result can be generalized for $\pv$-s, for which its first $k$ coordinates are $\infty$, but the others are strongly between $1$ and $\infty$, that is, for $\pv$-s, with
\begin{equation}\label{eq52}
\pv = (\infty,\infty,\ldots,\infty,p_{k+1},\ldots,p_{d}), \quad 1 < p_{k+1}, \ldots, p_{d} < \infty
\end{equation}
for some $k \in \{ 1, \ldots, d \}$.
\end{remark}

Now, we can generalize the well-known Doob's inequality. Using the previous theorem, we get that the maximal operator $M$ is bounded on $L_{\overrightarrow{p}}$ in case $1 < \pv < \infty$.
\begin{thm}\label{M korlatossaga}
If $1 < \pv < \infty$ or $\pv$ satisfies \eqref{eq52}, then the maximal operator $M$ is bounded on $L_{\overrightarrow{p}}$, that is for all $f \in L_{\overrightarrow{p}}$,
$$
\left\| Mf \right\|_{\overrightarrow{p}} \leq c \left\| f \right\|_{\overrightarrow{p}}.
$$
\end{thm}
\begin{proof}
It is clear that $Mf \leq \widetilde{M}f = M_{d} \circ \cdots \circ M_{1} f$, therefore by Theorem \ref{tbagby} and Remark \ref{remark1},
\begin{eqnarray*}
\left\|Mf\right\|_{\overrightarrow{p}} &\leq& \left\| M_{d} \circ M_{d-1} \circ \ldots \circ M_{1}f \right\|_{\overrightarrow{p}} \leq c \left\| M_{d-1} \circ \ldots \circ M_{1}f \right\|_{\overrightarrow{p}}\\
&\leq& c \left\| M_{d-2} \circ \ldots \circ M_{1}f \right\|_{\overrightarrow{p}} \leq \cdots \leq c \, \|f\|_{\pv}
\end{eqnarray*}
and the proof is complete.
\end{proof}
The following corollary is well-known for classical Hardy spaces with $\pv = (p,\ldots, p)$.
\begin{cor}
If $1 < \pv < \infty$, or $\pv$ satisfies \eqref{eq52}, then $H_{\pv}^{M}$ is equivalent to $L_{\pv}$.
\end{cor}

\begin{thm}
Theorem \ref{M korlatossaga} is not true for all $1 < \pv \leq \infty$.
\end{thm}
\begin{proof}
We prove the theorem for two dimensions and for the exponent $\pv := (p,\infty)$, where $1 < p < \infty$. The proof is similar for higher dimensions. Let us define the following sequence of functions
$$
f_{n}(x,y) := \sum_{k=1}^{n} 2^{k/p} \chi_{[2^{-k},2^{-k+1})^2}(x,y) \quad ( (x,y) \in [0,1) \times [0,1) ).
$$
Then for an arbitrary fixed $y \in [2^{-k},2^{-k+1})$ $(k=1,\ldots,n)$,
$$
\int_{[0,1)} |f_{n}(x,y)|^p \, dx = 2^{k} \, \frac{1}{2^{k}} = 1
$$
and for all fixed $y \notin [2^{-n},1)$, the previous integral is $0$. From this follows that for all $n \in \N$,
$$
\|f_{n}\|_{(p,\infty)} = \sup_{y \in [0,1)} \left( \int_{[0,1)} |f(x,y)|^p \, dx \right)^{1/p} = 1.
$$
At the same time, for $x \in [2^{-k},2^{-k+1})$ $(k=1,\ldots,n)$ and $y \in [0,2^{-n})$,
$$
Mf_{n}(x,y) \geq \frac{1}{\left|[0,2^{-k+1})^2 \right|} \int_{[0,2^{-k+1})^2} |f_{n}(u,v)| \, dudv \geq \frac{1}{2^{-2k+2}} \, 2^{k/p} \, 2^{-2k} = \frac{2^{k/p}}{4}.
$$
Hence we get that for all $y \in [0,2^{-n})$,
\begin{eqnarray*}
\int_{[0,1)} |Mf_{n}(x,y)|^p \, dx &\geq& \sum_{k=1}^{n} \int_{[2^{-k},2^{-k+1})} |Mf_{n}(x,y)|^p \, dx \\
&\geq& \sum_{k=1}^{n} \frac{2^{k}}{4} \, 2^{-k} = \frac{n}{4} \to \infty \quad \left(n \to \infty \right)
\end{eqnarray*}
and therefore
$$
\|Mf_{n}\|_{(p,\infty)} \to \infty \quad \left(n \to \infty \right),
$$
which means, that $M$ is not bounded on $L_{(p,\infty)}$.
\end{proof}

\begin{remark}
This counterexample proves also that $M_{2}$ is not bounded on $L_{(p,\infty)}$. Moreover, the counterexample shows also that the classical Hardy-Littlewood maximal operator considered in Huang et al. \cite{Huang2019} is not bounded on $L_{(p, \infty)}$ (cf. Lemma 3.5 in \cite{Huang2019} and Lemma 4.8 in \cite{Nogayama2019}).
\end{remark}

\section{Atomic decomposition}\label{atomic decomp}

First of all we need the definition of the \emph{atoms}. For $\overrightarrow{p}$, a measurable function $a$ is called an \emph{(s,$\overrightarrow{p}$)-atom} (or \emph{(S,$\overrightarrow{p}$)-atom} or \emph{(M,$\overrightarrow{p}$)-atom}) if there exists a stopping time $\tau$ such that
\begin{enumerate}
\item $\mathbb{E}_n a = 0$ for all $n \leq \tau$,
\item $\left\| s(a) \chi_{ \left\{ \tau < \infty \right\} }\right\|_{\infty} \leq  \left\| \chi_{ \left\{\tau < \infty\right\}} \right\|_{\overrightarrow{p}}^{-1}$\\
    (or $\left\| S(a) \chi_{ \left\{ \tau < \infty \right\} }\right\|_{\infty} \leq \left\| \chi_{ \left\{\tau < \infty\right\}} \right\|_{\overrightarrow{p}}^{-1}$, or $\left\| M(a) \chi_{ \left\{ \tau < \infty \right\} }\right\|_{\infty} \leq \left\| \chi_{ \left\{\tau < \infty\right\}} \right\|_{\overrightarrow{p}}^{-1}$, respectively).
\end{enumerate}

Now we can give the atomic decomposition of the space $H_{\pv}^{s}$.

\begin{thm}\label{atomic decomp of s}
Let $0 < \pv < \infty$. A martingale $f = \left(f_{n}\right)_{n \in \N} \in H_{\overrightarrow{p}}^{s}$ if and only if there exist a sequence $(a^{k})_{k \in \Z}$ of $(s,\overrightarrow{p})$-atoms and a sequence $(\mu_{k})_{k \in \Z}$ of real numbers such that
\begin{equation}\label{felbontas}
f_{n} = \sum_{k \in \Z} \mu_{k} \E_{n}a^{k} \qquad \mbox{a. e.} \quad \left( n \in \N \right)
\end{equation}
and
\begin{eqnarray}\label{inf Hps}
\left\| f \right\|_{H_{\overrightarrow{p}}^{s}} \sim \inf \left\| \left( \sum_{k \in \Z} \left( \frac{\mu_{k} \chi_{ \left\{\tau_{k} < \infty \right\}}}{ \left\| \chi_{ \left\{\tau_{k} < \infty \right\}} \right\|_{\overrightarrow{p}}} \right)^{t} \right)^{1/t} \right\|_{\overrightarrow{p}},
\end{eqnarray}
where $0 < t \leq 1$ and the infimum is taken over all decompositions of the form \eqref{felbontas}.
\end{thm}
\begin{proof}
Let $f \in H_{\overrightarrow{p}}^{s}$ and let us define the following stopping time:
$$
\tau_{k} := \inf\left\{ n \in \N : s_{n+1}(f) > 2^k \right\}.
$$
Obviously $f_{n}$ can be written in the form
$$
f_{n} = \sum_{k \in \Z} \left( f_{n}^{\tau_{k+1}} - f_{n}^{\tau_{k}} \right).
$$
Let
\begin{equation}\label{eq51}
\mu_{k} := 3 \cdot 2^k \left\| \chi_{ \left\{\tau_{k} < \infty \right\} } \right\|_{\overrightarrow{p}}, \quad \mbox{and} \quad a_{n}^{k} := \frac{f_{n}^{\tau_{k+1}} - f_{n}^{\tau_{k}}}{\mu_{k}}.
\end{equation}
If $\mu_{k} = 0$,  then let $a_{n}^k = 0$. If $n \leq \tau_{k}$, then $a_{n}^k=0$ and naturally
$$
f_{n} = \sum_{k \in \Z} \mu_{k} a_{n}^k.
$$
Moreover, $(a_{n}^{k})$ is $L_{2}$-bounded (see \cite{wk}), therefore there exists $a^k \in L_{2}$ such that $\mathbb{E}_{n} a^k = a_{n}^k$. Because of $s \left(f^{\tau_{k}}\right) = s_{\tau_{k}} (f) \leq 2^k$, we have that
$$
s \left(a^k\right) \leq \frac{s\left(f^{\tau_{k+1}}\right) + s \left(f^{\tau_{k}}\right)}{\mu_k} \leq \left\| \chi_{ \left\{ \tau_{k} < \infty \right\} } \right\|_{\overrightarrow{p}}^{-1},
$$
thus $a^k$ is an $(s,\overrightarrow{p})$-atom.

Since $ s^2(f-f^{\tau_k}) = s^2(f) - s^2 \left(f^{\tau_k}\right)$, thus $s(f-f^{\tau_k}) \leq s(f)$ and $s \left(f^{\tau_k}\right) \leq s(f)$. Using that $\lim_{k\to \infty} s \left(f- f^{\tau_{k}} \right) = \lim_{k \to \infty} s \left(f^{\tau_{k}}\right) = 0$ almost everywhere, by the dominated convergence theorem (see e.g. \cite{Benedek1961}) we get that
$$
\left\| f - \sum_{k=-l}^{m} \mu_k a^k \right\|_{H_{\pv}^s} \leq \left\| f - f^{\tau_{m+1}} \right\|_{H_{\pv}^s} + \left\| f^{\tau_{-l}} \right\|_{H_{\pv}^s} \to 0 \quad \left(l,m \to \infty \right),
$$
this means that $f = \sum_{k \in \Z} \mu_{k} a^k$ in the $H_{\overrightarrow{p}}^s$-norm.

Denote $\mathcal{O}_{k} := \left\{ \tau_{k} < \infty \right\} = \left\{ s(f) > 2^k \right\}$. Then for all $k \in \Z$, $\mathcal{O}_{k+1} \subset \mathcal{O}_{k}$. Moreover, for all $x \in \Omega$ and for all $0 < t \leq 1$,
$$
\sum_{k \in \Z} \left( 3 \cdot 2^k  \chi_{\mathcal{O}_{k}}(x) \right)^t \leq C \left( \sum_{k \in \Z} 3 \cdot 2^k \chi_{\mathcal{O}_{k} \setminus \mathcal{O}_{k+1}} \left(x\right) \right)^t.
$$
Since the sets $\mathcal{O}_{k} \setminus \mathcal{O}_{k+1}$ are disjoint, we have
\begin{eqnarray*}
\left\| \left( \sum_{k \in \Z} \left( \frac{\mu_{k} \chi_{ \left\{\tau_{k} < \infty \right\}}}{ \left\| \chi_{ \left\{\tau_{k} < \infty \right\}} \right\|_{\overrightarrow{p}}} \right)^{t} \right)^{1/t} \right\|_{\overrightarrow{p}} &=& \left\| \left( \sum_{k \in \Z} \left(3 \cdot 2^k \chi_{ \left\{\tau_{k} < \infty \right\}} \right)^{t} \right)^{1/t} \right\|_{\overrightarrow{p}}\\
&\leq& C \left\| \sum_{k \in \Z} 3 \cdot 2^k \chi_{ \mathcal{O}_{k} \setminus \mathcal{O}_{k+1}}   \right\|_{\overrightarrow{p}} \\
&\leq& C \left\| \sum_{k \in \Z} s(f) \chi_{ \mathcal{O}_{k} \setminus \mathcal{O}_{k+1}}   \right\|_{\overrightarrow{p}} \\
&=& C \left\| s(f) \right\|_{\overrightarrow{p}}.
\end{eqnarray*}

Conversely, if $f$ has a decomposition of the form \eqref{felbontas}, then
$$
s(f) \leq \sum_{k \in \Z} \mu_{k} s(a^k)  \leq \sum_{k \in \Z} \mu_{k}  \frac{ \chi_{ \left\{\tau_{k} < \infty \right\} } }{\left\| \chi_{ \left\{\tau_{k} < \infty \right\} } \right\|_{\overrightarrow{p}} },
$$
and so for all $0 < t \leq 1$,
\begin{eqnarray*}
\left\|f \right\|_{H_{\overrightarrow{p}}^{s}} \leq \left\| \sum_{k \in \Z} \mu_{k}  \frac{ \chi_{ \left\{\tau_{k} < \infty \right\} } }{\left\| \chi_{ \left\{\tau_{k} < \infty \right\} } \right\|_{\overrightarrow{p}} } \right\|_{\overrightarrow{p}} \leq \left\| \left( \sum_{k \in \Z} \left(   \frac{\mu_{k} \chi_{ \left\{\tau_{k} < \infty \right\} } }{\left\| \chi_{ \left\{\tau_{k} < \infty \right\} } \right\|_{\overrightarrow{p}} } \right)^t \right) ^{1/t} \right\|_{\overrightarrow{p}},
\end{eqnarray*}
which proves the theorem.
\end{proof}

For the classical martingale Hardy space $H_{p}^s$, this result is due to the second author (see \cite{wk}). For the spaces $\cQ_{\pv}$ and $\cP_{\pv}$ we can give similar decompositions.

\begin{thm}\label{atomos-felb-M}
Let $0 < \pv < \infty$. A martingale $f = \left(f_{n}\right)_{n \in \N} \in \mathcal{P}_{\overrightarrow{p}}$ (or $\in \mathcal{Q}_{\overrightarrow{p}}$) if and only if there exist a sequence $(a^{k})_{k \in \Z}$ of $(M,\overrightarrow{p})$-atoms (or $(S,\overrightarrow{p})$-atoms) and a sequence $(\mu_{k})_{k \in \Z}$ of real numbers such that \eqref{felbontas} holds and
$$
\left\| f \right\|_{\mathcal{P}_{\overrightarrow{p}}} \left(\mbox{or} \, \left\| f \right\|_{\mathcal{Q}_{\overrightarrow{p}}} \right) \sim \inf \left\| \left( \sum_{k \in \Z} \left( \frac{\mu_{k} \chi_{ \left\{\tau_{k} < \infty \right\}}}{ \left\| \chi_{ \left\{\tau_{k} < \infty \right\}} \right\|_{\overrightarrow{p}}} \right)^{t} \right)^{1/t} \right\|_{\overrightarrow{p}},
$$
where $0 < t \leq 1$ and the infimum is taken over all decompositions of the form \eqref{felbontas}.
\end{thm}
\begin{proof}
Let $f \in \mathcal{P}_{\overrightarrow{p}}$ and $ \left(\lambda_{n}\right)_{n \in \N}$ be a sequence such that $|f_{n}| \leq \lambda_{n-1}$ and $\lambda_{\infty} = \sup_{n} \lambda_{n} \in L_{\overrightarrow{p}}$. Let the stopping time $\tau_{k}$ be defined by
$$
\tau_{k} := \inf \left\{ n \in \N : \lambda_{n} > 2^k \right\}
$$
and $\mu_{k}$ and $a_{n}^k$ be given by \eqref{eq51}. Then again $f_{n} = \sum_{k \in \Z} \mu_{k} a_{n}^k$ and we can prove as before that
$$
\left\| \left( \sum_{k \in \Z} \left( \frac{\mu_{k} \chi_{ \left\{\tau_{k} < \infty \right\}}}{ \left\| \chi_{ \left\{\tau_{k} < \infty \right\}} \right\|_{\overrightarrow{p}}} \right)^{t} \right)^{1/t} \right\|_{\overrightarrow{p}} \leq C \left\| f \right\|_{\mathcal{P}_{\overrightarrow{p}}}.
$$

Conversely, assume that for some $\mu_{k}$ and $a_{n}^{k}$, the martingale $(f_{n})$ can be written in the form \eqref{felbontas}. For $n \in \N$, let us define
$$
\lambda_{n} := \sum_{k \in \Z} \mu_{k} \chi_{ \{ \tau_{k} \leq n \} } \left\| M \left(a^k \right) \right\|_{\infty}.
$$
It is clear, that $(\lambda_{n})$ is a nonnegative adapted sequence and for all $n \in \N$, $|f_{n}| \leq \lambda_{n-1}$. Therefore, for all $0 < t \leq 1$,
$$
\|f\|_{\cP_{\pv}} \leq \|\lambda_{\infty}\|_{\pv} \leq \left\| \sum_{k \in \Z} \frac{\mu_{k} \chi_{ \{ \tau_{k} < \infty \} } }{ \| \chi_{ \{ \tau_{k} < \infty \} } \|_{\pv} } \right\|_{\pv} \leq \left\| \left( \sum_{k \in \Z} \left( \frac{\mu_{k} \chi_{ \{ \tau_{k} < \infty \} } }{ \| \chi_{ \{ \tau_{k} < \infty \} } \|_{\pv} } \right)^t \right)^{1/t} \right\|_{\pv}.
$$
The case of $\cQ_{\pv}$ is similar.
\end{proof}

The stochastic basis $(\cF_{n})$ is said to be \emph{regular}, if there exists $R > 0$ such that for all $A \in \cF_{n}$ there exists $B \in \cF_{n-1}$ for which $A \subset B$ and $\P(B) \leq R \, \P(A)$. If the stochastic basis is regular, then atomic decomposition can also be proved for the remainder two martingale Hardy spaces, $H_{\overrightarrow{p}}^{M}$ and $H_{\overrightarrow{p}}^{S}$.

\begin{thm}\label{atomos-felb-M_and_S}
Let $0 < \pv < \infty$ and suppose that the stochastic basis $(\cF_{n})$ is regular. A martingale $f = \left(f_{n}\right)_{n \in \N} \in H_{\overrightarrow{p}}^{M}$ (or $\in H_{\overrightarrow{p}}^{S}$) if and only if there exist a sequence $(a^{k})_{k \in \Z}$ of $(M,\overrightarrow{p})$-atoms (or $(S,\overrightarrow{p})$-atoms) and a sequence $(\mu_{k})_{k \in \Z}$ of real numbers such that \eqref{felbontas} holds and
$$
\left\| f \right\|_{H_{\overrightarrow{p}}^{M}} \left(\mbox{or} \, \left\| f \right\|_{H_{\overrightarrow{p}}^{S}} \right) \sim \inf \left\| \left( \sum_{k \in \Z} \left( \frac{\mu_{k} \chi_{ \left\{\tau_{k} < \infty \right\}}}{ \left\| \chi_{ \left\{\tau_{k} < \infty \right\}} \right\|_{\overrightarrow{p}}} \right)^{t} \right)^{1/t} \right\|_{\overrightarrow{p}},
$$
where $0 < t < \min\{p_{1},\ldots,p_{d},1 \}$ and the infimum is taken over all decompositions of the form \eqref{felbontas}.
\end{thm}
\begin{proof}
We will prove the theorem only for $H_{\pv}^{M}$. The case of $H_{\pv}^S$ is similar. Suppose that $f \in H_{\pv}^{M}$ and for $k \in \Z$, let us define the stopping time $\varrho_{k} := \inf \{ n \in \N : |f_{n}| > 2^k \}$. Since $(\cF_{n})$ is regular and the set $I_{k,j} := \{ \varrho_{k} = j \} \in \cF_{j}$, there exist $\overline{I}_{k,j} \in \cF_{j-1}$ such that $I_{k,j} \subset \overline{I}_{k,j}$ and $\P(\overline{I}_{k,j}) \leq R \, \P(I_{k,j})$. Let us define an other stopping time
$$
\tau_{k}(x) := \inf\{ n \in \N : x \in \overline{I}_{k,n+1} \} \quad \left( x \in \Omega, \, k \in \Z \right).
$$
Then $\tau_{k}$ is non-decreasing and using Lemma \ref{megallasi-idok},
$$
\| \chi_{ \{ \tau_{k} < \infty \}} \|_{\pv} \leq C \, \| \chi_{ \{ \varrho_{k} < \infty \}} \|_{\pv} = \| \chi_{Mf > 2^k}\|_{\pv} \leq 2^{-k} \, \|Mf\|_{\pv}  \to 0 \quad \left( k \to \infty \right),
$$
which means that $\lim_{k \to \infty} \P(\{ \tau_{k} < \infty \}) = 0$, that is $\lim_{k \to \infty} \tau_{k} = \infty$ almost everywhere and therefore
$$
\lim_{k \to \infty} f_{n}^{\tau_{k}} = f_{n} \quad \left( a.e., \, n \in \N \right).
$$
Let $\mu_{k}$ and $a_{n}^{k}$ be defined again as in \eqref{eq51}. Then $a^k = (a_{n}^{k})$ is again an $(M,\pv)$-atom. For all $0 < t < \min\{p_{1},\ldots,p_{d},1 \}$,
$$
Z := \left\| \left[ \sum_{k \in \Z} \left( \sum_{j \in \N} \frac{\mu_k \chi_{\overline{I}_{k,j}} }{\| \chi_{\{\tau_{k} < \infty \}} \|_{\pv}} \right)^t  \right]^{1/t}  \right\|_{\pv} \leq \left\| \sum_{k \in \Z}  \sum_{j \in \N} \left( 3 \cdot 2^k  \right)^t \chi_{\overline{I}_{k,j}}  \right\|_{\pv/t}^{1/t},
$$
and therefore, using Lemma \ref{norma-sup} and Hölder's inequality,
\begin{eqnarray*}
Z^t &=& \int_{\Omega} \sum_{k \in \Z} \sum_{j \in N} \left(3 \cdot 2^k\right)^t \chi_{\overline{I}_{k,j}} g \, d\P \\
&\leq& \sum_{k \in \Z} \sum_{j \in N} \left(3 \cdot 2^k\right)^t \P(\overline{I}_{k,j})^{1/r} \left( \int_{\Omega} \chi_{\overline{I}_{k,j}} \, g^{r'} \, d\P \right)^{1/r'} \\
&=& \sum_{k \in \Z} \sum_{j \in N} \left(3 \cdot 2^k\right)^t \P(\overline{I}_{k,j}) \left( \frac{1}{\P(\overline{I}_{k,j})} \int_{\overline{I}_{k,j}} g^{r'} \, d\P \right)^{1/r'},
\end{eqnarray*}
where $\|g\|_{(\pv/t)'} \leq 1$. The stochastic basis $(\cF_{n})$ is regular and by $(\pv/t)' > 1$, $M$ is bounded on $L_{(\pv/t)'}$, therefore
\begin{eqnarray*}
Z^t &\leq& R \, \sum_{k \in \Z} \sum_{j \in N} \left(3 \cdot 2^k\right)^t \P(I_{k,j}) \left( \frac{1}{\P(\overline{I}_{k,j})} \int_{\overline{I}_{k,j}} g^{r'} \, d\P \right)^{1/r'} \\
&\leq& R \, \sum_{k \in \Z} \sum_{j \in N} \left(3 \cdot 2^k\right)^t \int_{\Omega} \chi_{I_{k,j}} \left[ M(g^{r'})\right]^{1/r'} \, d\P \\
&\leq& R \left\| \sum_{k \in \Z} \sum_{j \in N} \left(3 \cdot 2^k\right)^t \chi_{I_{k,j}} \right\|_{\pv/t} \left\| \left[ M(g^{r'}) \right]^{1/r'} \right\|_{(\pv/t)'} \\
&\leq& C \, R \, \left\| \sum_{k \in \Z} \left(3 \cdot 2^k \chi_{ \{ \varrho_{k} < \infty \} } \right)^t \right\|_{\pv/t} \\
&=& C \left\| \left[ \sum_{k \in \Z} \left(3 \cdot 2^k \chi_{ \{ Mf > 2^k \} } \right)^t \right]^{1/t} \right\|_{\pv}^{t},
\end{eqnarray*}
where we have used that $\{ \varrho_{k} < \infty \} = \{ Mf > 2^k \}$. So we have that
$$
Z \leq C \left\| \left[ \sum_{k \in \Z} \left(3 \cdot 2^k \chi_{ \{ Mf > 2^k \} } \right)^t \right]^{1/t} \right\|_{\pv},
$$
where the right hand side can be estimated by $\|f\|_{H_{\pv}^M}$, similarly as in the proof of Theorem \ref{atomic decomp of s}.

Conversely, if $f$ has a decomposition of the form $\eqref{felbontas}$, then
$$
\|f\|_{H_{\pv}^M} \leq \left\| \sum_{k \in \Z} \mu_{k} \frac{\chi_{ \{ \tau_{k} < \infty \} } }{\| \chi_{ \{ \tau_{k} < \infty \} } \|_{\pv}} \right\|_{\pv} \leq \left\| \left[ \sum_{k \in \Z} \left( \mu_{k} \frac{\chi_{ \{ \tau_{k} < \infty \} } }{\| \chi_{ \{ \tau_{k} < \infty \} } \|_{\pv}} \right)^t \right]^{1/t} \right\|_{\pv}
$$
can be proved similarly as in Theorem \ref{atomic decomp of s}.
\end{proof}

\begin{lem}\label{megallasi-idok}
Let $0 < \pv < \infty$ and suppose that the stochastic basis $(\cF_{n})$ is regular. If $\varrho_{k}$ and $\tau_{k}$ are the stopping times defined in the proof of Theorem \ref{atomos-felb-M_and_S}, then
$$
\| \chi_{ \{ \tau_{k} < \infty \}} \|_{\pv} \leq C \, \| \chi_{ \{ \varrho_{k} < \infty \}} \|_{\pv}.
$$
\end{lem}
\begin{proof}
It is enough to prove that for some $0 < \varepsilon < \min\{ 1, p_{1}, \ldots, p_{d} \}$, the inequality
$$
\| \chi_{ \{ \tau_{k} < \infty \}} \|_{\pv/\varepsilon} \leq C \, \| \chi_{ \{ \varrho_{k} < \infty \}} \|_{\pv/\varepsilon}
$$
holds. Notice that
$$
\| \chi_{ \{ \tau_{k} < \infty \}} \|_{\pv/\varepsilon} = \left\| \sum_{n \in \N} \chi_{\overline{I}_{k,n}} \right\|_{\pv/\varepsilon} \quad \mbox{and} \quad \| \chi_{ \{ \varrho_{k} < \infty \}} \|_{\pv/\varepsilon} = \left\| \sum_{n \in \N} \chi_{I_{k,n}} \right\|_{\pv/\varepsilon}.
$$
By Lemma \ref{norma-sup}, there exists a function $g$ with $\|g\|_{(\pv/\varepsilon)'} \leq 1$, such that
$$
\| \chi_{ \{ \tau_{k} < \infty \}} \|_{\pv/\varepsilon} = \int_{\Omega} \sum_{n \in \N} \chi_{\overline{I}_{k,n}} \, g \, d\P.
$$
Using Hölder's inequality with $\pv/\varepsilon < r < \infty$ and the regularity of $(\cF_{n})$, we obtain
\begin{eqnarray*}
\| \chi_{ \{ \tau_{k} < \infty \}} \|_{\pv/\varepsilon} &\leq& \sum_{n \in \N} \P(\overline{I}_{k,n})^{1/r} \left( \int_{\Omega} \chi_{\overline{I}_{k,n}} \, g^{r'} \, d\P \right)^{1/r'}\\
&=& \sum_{n \in \N} \P(\overline{I}_{k,n}) \left(\frac{1}{\P(\overline{I}_{k,n})} \int_{\Omega} \chi_{\overline{I}_{k,n}} \, g^{r'} \, d\P \right)^{1/r'} \\
&\leq& R \, \sum_{n \in \N} \int_{\Omega} \chi_{I_{k,n}} \left[ M(g^{r'}) \right]^{1/r'} \, d\P.
\end{eqnarray*}
By Lemma \ref{holder}, we get that
\begin{eqnarray*}
\| \chi_{ \{ \tau_{k} < \infty \}} \|_{\pv/\varepsilon} &\leq& R \, \left\| \sum_{n \in \N} \chi_{I_{k,n}} \right\|_{\pv/\varepsilon} \left\| \left[ M(g^{r'}) \right]^{1/r'} \right\|_{\pv/\varepsilon}\\
&\leq& C \, R \, \| \chi_{ \{ \varrho_{k} < \infty \}} \|_{\pv/\varepsilon},
\end{eqnarray*}
where we have used that $(\pv/\varepsilon)' > 1$ and therefore $M$ is bounded on $L_{(\pv/\varepsilon)'}$. The proof is complete.
\end{proof}

\begin{cor}\label{ekvivalens}
If the stochastic basis $(\cF_{n})$ is regular, then
$$
H_{\pv}^{S} = \cQ_{\pv} \quad \mbox{and} \quad H_{\pv}^{M} = \cP_{\pv} \quad \left( 0 < \pv < \infty \right)
$$
with equivalent quasi-norms.
\end{cor}

\section{Martingale inequalities}\label{mart ineq}

We will prove the analogous version of the classical martingale inequalities (see e.g., Weisz \cite{wk})
for the five mixed martingale Hardy spaces. To this end, we need the following boundedness results.

Let $X$ be a martingale space, $Y$ be a measurable function space. Then the operator $U: X \to Y$ is called \emph{$\sigma$-sublinear operator}, if for any $\alpha \in \C$,
$$
\left| U \left( \sum_{k=1}^{\infty} f_{k} \right) \right| \leq \sum_{k=1}^{\infty} |U(f_{k})| \quad \mbox{and} \quad  \left| U\left( \alpha f \right) \right| = |\alpha| |U(f)|.
$$

\begin{thm}\label{T korl Hps-bol}
Let $0 < \pv < \infty$ and suppose that the $\sigma$-sublinear operator $T: H_{r}^{s} \to L_{r}$ is bounded, where $\pv = (p_{1}, \ldots, p_{d})$ and $r > p_{i}$ $(i=1,\ldots,d)$. If for all $(s,\overrightarrow{p})$-atom $a$
\begin{equation}\label{T-atom_s}
(Ta)\chi_{A} = T \left(a \chi_{A}\right) \quad \left(A \in \mathcal{F}_{\tau}\right),
\end{equation}
where $\tau$ is the stopping time associated with the $(s,\overrightarrow{p})$-atom $a$, then for all $f \in H_{\overrightarrow{p}}^{s}$,
$$
\left\| Tf \right\|_{\overrightarrow{p}} \leq C \left\|f\right\|_{H_{\overrightarrow{p}}^s}.
$$
\end{thm}
\begin{proof}
By the $\sigma$-sublinearity of $T$ and the atomic decomposition of $H_{\overrightarrow{p}}^{s}$ given in Theorem \ref{atomic decomp of s}, we have
$$
\left| Tf\right| \leq \sum_{k \in \Z} 3 \cdot 2^{k} \left\| \chi_{ \left\{\tau_{k} < \infty \right\} } \right\|_{\overrightarrow{p}} \left| T(a^k)\right|,
$$
If we choose $0 < t < \min\left\{ p_{1},\ldots,p_{d},1 \right\} \leq 1$, then
\begin{eqnarray*}
\left\|Tf\right\|_{\overrightarrow{p}} &\leq& \left\| \left[ \sum_{k \in \Z} \left( 3 \cdot 2^{k} \left\| \chi_{ \left\{\tau_{k} < \infty \right\} } \right\|_{\overrightarrow{p}} \left| T(a^k)\right| \right)^{t} \right]^{1/t} \right\|_{\overrightarrow{p}}\\
&=& \left\|  \sum_{k \in \Z} \left( 3 \cdot 2^{k} \right)^t \left\| \chi_{ \left\{\tau_{k} < \infty \right\} } \right\|_{\overrightarrow{p}}^t \left| T(a^k)\right|^t \right\|_{\overrightarrow{p}/t}^{1/t} =: Z.
\end{eqnarray*}
By Lemma \ref{norma-sup}, there exists a function $g \in L_{(\overrightarrow{p}/t)'}$ with $\|g\|_{(\overrightarrow{p}/t)'} \leq 1$, such that
$$
Z^t = \int_{\Omega} \sum_{k \in \Z} \left(3 \cdot 2^k\right)^t \left\| \chi_{ \left\{\tau_{k} < \infty \right\} } \right\|_{\overrightarrow{p}}^t \left| T(a^k)\right|^t g \, d\mathbb{P}.
$$
Since $\left\{ \tau_{k} < \infty \right\} \in \mathcal{F}_{\tau_{k}}$, using the fact that $a^k = a^k \chi_{ \left\{ \tau_{k} < \infty \right\}}$ and equation \eqref{T-atom_s}, we have $T \left(a^k \chi_{\left\{ \tau_{k} < \infty \right\}} \right) = T \left(a^k\right) \chi_{\left\{ \tau_{k} < \infty \right\}}$. Since $t < r$, the previous expression can be estimated by H\"older's inequality
\begin{eqnarray}\label{formula}
Z^t &\leq& \sum_{k \in \Z} \left(3 \cdot 2^k\right)^t \left\| \chi_{ \left\{\tau_{k} < \infty \right\} } \right\|_{\overrightarrow{p}}^t \int_{\Omega} \chi_{ \left\{ \tau_{k} < \infty \right\}} \mathbb{E}_{\tau_{k}} \left( \left| T(a^k) \right|^t g \right) \, d\mathbb{P} \n\\
&\leq& c \sum_{k \in \Z} \left(3 \cdot 2^k\right)^t \left\| \chi_{ \left\{\tau_{k} < \infty \right\} } \right\|_{\overrightarrow{p}}^t \int_{\Omega} \chi_{ \left\{ \tau_{k} < \infty \right\}} \n\\
& & \left[ \mathbb{E}_{\tau_{k}} \left( \left| T(a^k)\right|^r \right) \right]^{t/r} \left[ \mathbb{E}_{\tau_{k}} \left( \left| g\right|^{(r/t)'} \right) \right]^{1/(r/t)'} \, d\mathbb{P}.
\end{eqnarray}
Here, by the boundedness of $T$ and by the fact that $a^k$ is an $(s,\overrightarrow{p})$-atom, we get
\begin{eqnarray*}
\int_{A} |T(a^{k})|^r \, d\P &=& \int_{\Omega} |T(a^{k} \chi_{A})|^r  \, d\P  \leq \int_{\Omega} |s(a^{k} \chi_{A})|^r \, d\P \\
&\leq& \int_{A} |s(a^{k})|^r \, d\P \leq \|\chi_{\{ \tau_{k} < \infty \} }\|_{\pv}^{-r} \, \P(A),
\end{eqnarray*}
where $A \in \cF_{\tau_{k}}$. This implies that

\begin{eqnarray*}
\left[ \mathbb{E}_{\tau_{k}} \left( \left| T(a^k)\right|^r \right) \right]^{t/r} \leq \left\| \chi_{ \left\{ \tau_{k} < \infty \right\}} \right\|_{\overrightarrow{p}}^{-t}.
\end{eqnarray*}
Hence, \eqref{formula} can be estimated by
\begin{eqnarray*}
Z^t &\leq& c \int \sum_{k \in \Z} \left(3 \cdot 2^k\right)^t \chi_{ \left\{ \tau_{k} < \infty \right\}} \left[ \mathbb{E}_{\tau_{k}} \left( \left| g\right|^{(r/t)'} \right) \right]^{1/(r/t)'} \, d\mathbb{P} \\
&\leq& c \int \sum_{k \in \Z} \left(3 \cdot 2^k\right)^t \chi_{ \left\{ \tau_{k} < \infty \right\}} \left[ M \left( \left| g\right|^{(r/t)'} \right) \right]^{1/(r/t)'} \, d\mathbb{P}\\
&\leq& c \left\| \sum_{k \in \Z} \left(3 \cdot 2^k\right)^t \chi_{ \left\{ \tau_{k} < \infty \right\}} \right\|_{\overrightarrow{p}/t} \left\| \left[ M \left( \left| g\right|^{(r/t)'} \right) \right]^{1/(r/t)'} \right\|_{(\overrightarrow{p}/t)'}.
\end{eqnarray*}
Since $\pv < r$, therefore $(\pv/t)'/(r/t)' > 1$, so by the boundedness of $M$ (see Theorem \ref{M korlatossaga}), we obtain that
\begin{eqnarray*}
\left\| \left[ M \left( \left| g\right|^{(r/t)'} \right) \right]^{1/(r/t)'} \right\|_{(\overrightarrow{p}/t)'} &=& \left\| M \left( \left| g\right|^{(r/t)'} \right) \right\|_{(\overrightarrow{p}/t)' /(r/t)'}^{1/(r/t)'} \leq c \left\| \left| g\right|^{(r/t)'} \right\|_{(\overrightarrow{p}/t)' /(r/t)'}^{1/(r/t)'} \\
&=& c \left\|g \right\|_{(\overrightarrow{p}/t)'} \leq c.
\end{eqnarray*}
By \eqref{inf Hps}, we have
\begin{eqnarray*}
\left\|Tf\right\|_{\overrightarrow{p}} \leq c \left\| \sum_{k \in \Z} \left(3 \cdot 2\right)^t \chi_{ \left\{ \tau_{k} < \infty \right\}} \right\|_{\overrightarrow{p}/t}^{1/t}= c \left\| \left[ \sum_{k \in \Z} \left( \frac{ \mu_k \chi_{ \left\{ \tau_{k} < \infty \right\}}}{\left\| \chi_{ \left\{\tau_{k} < \infty \right\}} \right\|_{\overrightarrow{p}}} \right)^t  \right]^{1/t} \right\|_{\overrightarrow{p}} \leq c \left\| f \right\|_{H_{\overrightarrow{p}}^s}
\end{eqnarray*}
and we get that $T$ is bounded from $H_{\overrightarrow{p}}^s$ to $L_{\overrightarrow{p}}$.
\end{proof}

The following theorem can be proved similarly.

\begin{thm}\label{T korl HpS-bol es HpM-bol}
Let $0 < \pv < \infty$ and suppose that the $\sigma$-sublinear operator $T: \mathcal{Q}_{r} \to L_{r}$ (resp. $T: \mathcal{P}_{r} \to L_{r}$) is bounded, where $\pv=(p_{1},\ldots,p_{d})$ and $r > p_{i}$ $(i=1,\ldots,d)$. If all $(S,\overrightarrow{p})$-atoms (resp. $(M,\overrightarrow{p})$-atoms) $a$ satisfy \eqref{T-atom_s}, then for all $f \in \mathcal{Q}_{\overrightarrow{p}}$ (resp. $f \in \mathcal{P}_{\overrightarrow{p}}$),
$$
\left\| Tf \right\|_{\overrightarrow{p}} \leq C \left\|f\right\|_{\mathcal{Q}_{\overrightarrow{p}}} \quad \left(\mbox{resp.} \quad \left\|f\right\|_{\mathcal{P}_{\overrightarrow{p}}} \right).
$$
\end{thm}

It is easy to see that for all $(s,\overrightarrow{p})$-atoms $a$, $(S,\pv)$-atoms $a$ or $(M,\pv)$-atoms $a$ and $A \in \mathcal{F}_{\tau}$, $s(a \chi_{A}) = s(a) \chi_{A}$, $S(a \chi_{A}) = S(a) \chi_{A}$ and $M(a \chi_{A}) = M(a) \chi_{A}$. This means that the operators $s$, $S$ and $M$ satisfy condition \eqref{T-atom_s}.

Let $f \in H_{\overrightarrow{p}}^{s}$. The $\sigma$-sublinear operator $M$ is bounded from $H_{2}^{s}$ to $L_{2}$ (see e.g. Weisz \cite{wk}), that is $\left\|Mf\right\|_{2} \leq c \left\|f\right\|_{H_{2}^s}$. So we can apply Theorem \ref{T korl Hps-bol} with the choice $r=2$ and $\overrightarrow{p} := \left(p_{1},\ldots,p_{d}\right)$, where $p_{i} < 2$ and we get that
\begin{equation}\label{mart-ineq-1}
\|f\|_{H_{\overrightarrow{p}}^M} = \left\|M(f)\right\|_{\overrightarrow{p}} \leq c \left\|f\right\|_{H_{\overrightarrow{p}}^s} \quad \left( 0 < \overrightarrow{p} < 2 \right).
\end{equation}

The operator $S$ is also bounded from $H_{2}^s$ to $L_{2}$ (see \cite{wk}), hence using Theorem \ref{T korl Hps-bol} we obtain
\begin{equation}\label{mart-ineq-2}
\left\| f \right\|_{H_{\overrightarrow{p}}^S} \leq c \left\| f \right\|_{H_{\overrightarrow{p}}^s} \quad \left( 0 < \overrightarrow{p} < 2 \right).
\end{equation}

From the definition of the Hardy spaces it follows immediately that
\begin{equation}\label{mart-ineq-3}
\left\|f\right\|_{H_{\overrightarrow{p}}^M} \leq \left\|f\right\|_{\mathcal{P}_{\overrightarrow{p}}}, \quad \left\|f\right\|_{H_{\overrightarrow{p}}^S} \leq \left\|f\right\|_{\mathcal{Q}_{\overrightarrow{p}}} \quad \left(0 < \overrightarrow{p} < \infty \right).
\end{equation}

By the Burkholder-Gundy and Doob's inequality, for all $1 < r < \infty$, $\| S(f) \|_{r} \approx \left\|M(f)\right\|_{r} \approx \left\|f\right\|_{r}$ (see Weisz \cite{wk}). Using this, inequality \eqref{mart-ineq-3} and Theorem \ref{T korl HpS-bol es HpM-bol}, we have
\begin{equation}\label{mart-ineq-4}
\left\|f\right\|_{H_{\overrightarrow{p}}^S} \leq c \left\|f\right\|_{\mathcal{P}_{\overrightarrow{p}}} \quad \mbox{and} \quad \left\|f\right\|_{H_{\overrightarrow{p}}^M} \leq c \left\|f\right\|_{\mathcal{Q}_{\overrightarrow{p}}} \quad \left( 0 < \overrightarrow{p} < \infty \right).
\end{equation}

For $f = \left(f_{n}\right)_{n \in \N} \in \mathcal{Q}_{\overrightarrow{p}}$ there exists a sequence $(\lambda_{n})_{n \in \N}$ for which $S_{n}(f) \leq \lambda_{n-1}$ and $\lambda_{\infty} \in L_{\overrightarrow{p}}$. Using the inequality $|f_{n}| \leq M_{n-1}(f) + \lambda_{n-1}$ and the second inequality in \eqref{mart-ineq-4}, we get that
\begin{equation}\label{mart-ineq-5}
\left\|f\right\|_{\mathcal{P}_{\overrightarrow{p}}} \leq \left\|M(f)\right\|_{\overrightarrow{p}} + \left\|\lambda_{\infty}\right\|_{\overrightarrow{p}} \leq \left\|f\right\|_{H_{\overrightarrow{p}}^{M}} + c \, \left\|f\right\|_{\mathcal{Q}_{\overrightarrow{p}}} \leq c \left\|f\right\|_{\mathcal{Q}_{\overrightarrow{p}}}.
\end{equation}

Similarly, if $f = \left(f_{n}\right)_{n \in \N} \in \mathcal{P}_{\overrightarrow{p}}$, then $|f_{n}| \leq \lambda_{n-1}$ with a suitable sequence $(\lambda_{n})_{n \in \N}$ for which $\lambda_{\infty} \in L_{\overrightarrow{p}}$. Since
\begin{eqnarray*}
S_{n}(f) &=& \left( \sum_{k=0}^{n} \left| d_{k}f \right|^2 \right)^{1/2} \leq S_{n-1}(f) + |d_{n}f|  \leq S_{n-1}(f) + 2 \lambda_{n-1},
\end{eqnarray*}
using the first inequality in \eqref{mart-ineq-4}, we have that
\begin{equation}\label{mart-ineq-6}
\left\|f\right\|_{\mathcal{Q}_{\overrightarrow{p}}} \leq \left\|S(f)\right\|_{\overrightarrow{p}} + 2 \left\|\lambda_{\infty}\right\|_{\overrightarrow{p}} = \left\|f\right\|_{H_{\overrightarrow{p}}^S} + 2 \left\|f\right\|_{\mathcal{P}_{\overrightarrow{p}}} \leq c \left\|f\right\|_{\mathcal{P}_{\overrightarrow{p}}}
\end{equation}
for all $0 < \pv < \infty$.

From \cite{wk} Proposition 2.11 (ii), we get that the operator $s$ is bounded from $H_{r}^M$ to $L_r$ and from $H_{r}^S$ to $L_r$ if $2 \leq r < \infty$. Again, using Theorem \ref{T korl HpS-bol es HpM-bol}, we obtain
\begin{equation}\label{mart-ineq-7}
\left\|f\right\|_{H_{\overrightarrow{p}}^s} \leq c \left\|f\right\|_{\mathcal{P}_{\overrightarrow{p}}} \quad \mbox{and} \quad \left\|f\right\|_{H_{\overrightarrow{p}}^{s}} \leq c \left\|f\right\|_{\mathcal{Q}_{\overrightarrow{p}}} \quad \left(0 < \overrightarrow{p} < \infty \right).
\end{equation}

The inequalities \eqref{mart-ineq-1}, \eqref{mart-ineq-2}, \eqref{mart-ineq-3}, \eqref{mart-ineq-4}, \eqref{mart-ineq-5}, \eqref{mart-ineq-6} and \eqref{mart-ineq-7} are collected in the following corollary.

\begin{cor}
We have the following martingale inequalities:
\begin{enumerate}
\item
$$
\left\|f\right\|_{H_{\overrightarrow{p}}^M} \leq c \left\|f\right\|_{H_{\overrightarrow{p}}^s}, \quad \left\| f \right\|_{H_{\overrightarrow{p}}^S} \leq c \left\| f \right\|_{H_{\overrightarrow{p}}^s} \quad \left( 0 < \overrightarrow{p} < 2 \right).
$$
\item
$$
\left\|f\right\|_{H_{\overrightarrow{p}}^M} \leq \left\|f\right\|_{\mathcal{P}_{\overrightarrow{p}}}, \quad \left\|f\right\|_{H_{\overrightarrow{p}}^S} \leq \left\|f\right\|_{\mathcal{Q}_{\overrightarrow{p}}} \quad \left(0 < \overrightarrow{p} < \infty \right).
$$
\item
\begin{equation*}\label{eq400}
\left\|f\right\|_{H_{\overrightarrow{p}}^S} \leq c \left\|f\right\|_{\mathcal{P}_{\overrightarrow{p}}}, \quad  \left\|f\right\|_{H_{\overrightarrow{p}}^M} \leq c \left\|f\right\|_{\mathcal{Q}_{\overrightarrow{p}}} \quad \left( 0 < \overrightarrow{p} < \infty \right).
\end{equation*}
\item
\begin{equation*}\label{eq130}
\left\|f\right\|_{\mathcal{P}_{\overrightarrow{p}}} \leq c \left\|f\right\|_{\mathcal{Q}_{\overrightarrow{p}}}, \quad \left\|f\right\|_{\mathcal{Q}_{\overrightarrow{p}}} \leq c \left\|f\right\|_{\mathcal{P}_{\overrightarrow{p}}} \quad \left(0 < \overrightarrow{p} < \infty \right).
\end{equation*}
\item
\begin{equation}\label{s-pp-s-qp}
\left\|f\right\|_{H_{\overrightarrow{p}}^s} \leq c \left\|f\right\|_{\mathcal{P}_{\overrightarrow{p}}} \quad \mbox{and} \quad \left\|f\right\|_{H_{\overrightarrow{p}}^{s}} \leq c \left\|f\right\|_{\mathcal{Q}_{\overrightarrow{p}}} \quad \left(0 < \overrightarrow{p} < \infty \right).
\end{equation}
\end{enumerate}
\end{cor}

\begin{thm}\label{t3}
    If the stochastic basis $(\cF_{n})$ is regular, then the five Hardy spaces are equivalent, that is
    $$
    H_{\pv}^S = \cQ_{\pv} = \cP_{\pv} = H_{\pv}^M = H_{\pv}^s \qquad \left(0 < \pv < \infty \right)
    $$
    with equivalent quasi-norms.
\end{thm}

\begin{proof}
We know (see e.g. Weisz \cite{wk}) that $S_{n}(f) \leq R^{1/2} s_{n}(f)$ and from this, it follows that $\|f\|_{H_{\pv}^S} \leq C \|f\|_{H_{\pv}^s}$. Using the definition of $\cQ_{\pv}$ and the fact that $s_{n}(f) \in \cF_{n-1}$ we get
$$
\|f\|_{\cQ_{\pv}} \leq C \|s(f)\|_{\pv} = C \|f\|_{H_{\pv}^s}.
$$
By inequalities \eqref{s-pp-s-qp}, we obtain that $\cQ_{\pv} = H_{\pv}^s$. Since the stochastic basis $(\cF_{n})$ is regular, the theorem follows from Corollary \ref{ekvivalens} and from \eqref{mart-ineq-6}.
\end{proof}

\begin{thm}\label{vect-ineq}
Suppose that $1 < \pv < \infty$, or
\begin{equation}\label{eq71}
\pv=(1,\ldots,1,p_{k+1},\ldots,p_{d}), \quad 1 < p_{k+1}, \ldots, p_{d} < \infty
\end{equation}
for some $k \in \{1, \ldots, d\}$. Then for all $(f_{n})_{n \in \N}$ non-negative, measurable function sequence
$$
\left\| \sum_{n \in \N} \mathbb{E}_{n} \left(f_{n}\right) \right\|_{\pv} \leq c \left\| \sum_{n \in \N} f_{n} \right\|_{\pv}.
$$
\end{thm}
\begin{proof}
From Lemma \ref{norma-sup}, we know that there exists a function $g \in L_{\vpc}$ with $\|g\|_{\vpc} \leq 1$ such that
$$
\left\| \sum_{n \in \N} \mathbb{E}_{n} \left(f_{n}\right) \right\|_{\pv} = \int_{\Omega} \sum_{n \in \N} \mathbb{E}_{n} \left(f_{n}\right) g \, d\mathbb{P}.
$$
Since $\mathbb{E}_{n} \left(f_{n}\right)$ is $\mathcal{F}_{n}$-measurable, we obtain
\begin{eqnarray*}\label{holder elott}
\int_{\Omega} \sum_{n \in \N} \mathbb{E}_{n} \left(f_{n}\right) g \, d\mathbb{P} \leq \sum_{n \in \N} \int_{\Omega}  f_{n} M(g) \, d\mathbb{P} = \int_{\Omega}  \sum_{n \in \N} f_{n} M(g) \, d\mathbb{P}.
\end{eqnarray*}
Using H\"older's inequality and Theorem \ref{M korlatossaga}, we have
\begin{eqnarray*}
\left\| \sum_{n \in \N} \E_{n}(f_{n}) \right\|_{\pv} \leq c \left\| \sum_{n \in \N} f_{n} \right\|_{\pv} \left\| Mg \right\|_{\vpc} \leq c \, \left\| \sum_{n \in \N} f_{n} \right\|_{\pv} \left\|g\right\|_{\vpc} \leq c \, \left\| \sum_{n \in \N} f_{n} \right\|_{\pv}
\end{eqnarray*}
and the proof is complete.
\end{proof}

As an application of the previous theorem, we get the following martingale inequality.
\begin{cor}
If $2 < \pv < \infty$, or $\pv/2$ satisfies \eqref{eq71}, then
$$
\left\|f\right\|_{H_{\pv}^s} \leq c \left\|f\right\|_{H_{\pv}^S}.
$$
\end{cor}
\begin{proof}
Indeed, using Theorem \ref{vect-ineq} with the choice $f_{n} := \left| d_{n} f\right|^2$, we have
\begin{eqnarray*}
\left\| \left( \sum_{n \in \N} \mathbb{E}_{n-1} \left| d_{n}f\right|^2 \right)^{1/2} \right\|_{\pv} = \left\| \sum_{n \in \N} \mathbb{E}_{n-1} \left| d_{n}f\right|^2 \right\|_{\pv/2}^{1/2} \leq c \left\| \sum_{n \in \N} \left| d_{n}f\right|^2 \right\|_{\pv/2}^{1/2} = c \left\| S(f)\right\|_{\pv},
\end{eqnarray*}
which finishes the proof.
\end{proof}
To prove the Burkholder-Davis-Gundy inequality, we introduce the norm
$$
\|f\|_{\cG_{\pv}} := \left\| \sum_{n \in \N} |d_{n} f|  \right\|_{\pv}.
$$

\begin{lem}\label{lemma}
Suppose that $1 < \pv < \infty$ or $\pv$ satisfies \eqref{eq71}. If $f \in H_{\pv}^{S}$, then there exists $h \in \cG_{\pv}$ and $g \in \cQ_{\pv}$ such that $f = h + g$ and
$$
\|h\|_{\cG_{\pv}} \leq c \, \|f\|_{H_{\pv}^{S}} \qquad \mbox{and} \qquad \|g\|_{\cQ_{\pv}} \leq c \, \|f\|_{H_{\pv}^{S}}.
$$
\end{lem}

\begin{proof}
Let $f \in H_{\pv}^{S}$ and let $(\lambda_{n})$ be an adapted, non-decreasing sequence such that $\lambda_0=0$, $S_{n}f \leq \lambda_{n}$ and $\lambda_{\infty} \in L_{\pv}$. Let us define the functions
\begin{eqnarray*}
h_n &:=& \sum_{k = 1}^{n} \left( d_{k}f \chi_{ \{ \lambda_{k} > 2 \lambda_{k-1} \} } - \E_{k-1} \left( d_{k}f \chi_{ \{ \lambda_{k} > 2 \lambda_{k-1} \}} \right) \right),  \\
g_n &:=& \sum_{k = 1}^{n} \left( d_{k}f \chi_{ \{ \lambda_{k} \leq 2 \lambda_{k-1} \} } - \E_{k-1} \left( d_{k}f \chi_{ \{ \lambda_{k} \leq 2 \lambda_{k-1} \}} \right) \right).
\end{eqnarray*}
Then $f_n=h_n+g_n$ $(n \in \N)$. On the set $\{\lambda_k>2\lambda_{k-1}\}$,
we have $\lambda_k < 2(\lambda_k - \lambda_{k-1})$, henceforth
$$
|d_kf| \chi_{\{\lambda_k>2\lambda_{k-1}\}} \le\lambda_k \chi_{\{\lambda_k>2\lambda_{k-1}\}} \le 2(\lambda_k - \lambda_{k-1})
$$
and so
$$
\sum_{k=1}^{n} |d_{k}h|  \leq 2 \lambda_{n} + 2 \sum_{k=1}^{n} \E_{k-1}\left( \lambda_{k} - \lambda_{k-1} \right).
$$
Using Theorem \ref{vect-ineq}, we have that
$$
\|h\|_{\cG_{\pv}} \leq 2 \|\lambda_{\infty}\|_{\pv} + c \, \left\| \sum_{n \in \N} (\lambda_{n+1} - \lambda_{n}) \right\|_{\pv} \leq c \, \|\lambda_{\infty}\|_{\pv}.
$$
At the same time,

$$
|d_kf| \chi_{\{\lambda_k\le 2\lambda_{k-1}\}} \le \lambda_k \chi_{\{\lambda_k\le 2\lambda_{k-1}\}} \le 2\lambda_{k-1},
$$
which implies that $|d_kg| \le 4\lambda_{k-1}$. Then
\begin{align*}
	S_n(g) &\le S_{n-1}(g)+|d_ng| \le S_{n-1}(f)+S_{n-1}(h)+4\lambda_{n-1} \cr
&\le \lambda_{n-1} + 2\lambda_{n-1} +
2 \sum_{k=1}^{n-1} E_{k-1} (\lambda_k - \lambda_{k-1}) + 4\lambda_{n-1}
\end{align*}
and therefore
$$
\|g\|_{\cQ_{\pv}} \leq c \, \|\lambda_{\infty}\|_{\pv}.
$$
Choosing $\lambda_{n} := S_{n}(f)$, we get that
$$
\|h\|_{\cG_{\pv}} \leq c \, \|f\|_{H_{\pv}^{S}} \quad \mbox{and} \quad \|g\|_{\cQ_{\pv}} \leq c \, \|f\|_{H_{\pv}^{S}},
$$
which proves the theorem.
\end{proof}

A similar lemma can be proved for $H_{\pv}^{M}$ in the same way.

\begin{lem}\label{l1}
Suppose that $1 < \pv < \infty$ or $\pv$ satisfies \eqref{eq71}. If $f \in H_{\pv}^{M}$, then there exists $h \in \cG_{\pv}$ and $g \in \cP_{\pv}$ such that $f = h + g$ and
$$
\|h\|_{\cG_{\pv}} \leq c \, \|f\|_{H_{\pv}^{M}} \qquad \mbox{and} \qquad \|g\|_{\cP_{\pv}} \leq c \, \|f\|_{H_{\pv}^{M}}.
$$
\end{lem}

Now we are ready to generalize the well known Burkholder-Davis-Gundy inequality.

\begin{thm}\label{BDG}
If $1 < \pv < \infty$ or $\pv$ satisfies \eqref{eq71}, then the spaces $H_{\pv}^{S}$ and $H_{\pv}^{M}$ are equivalent, that is
$$
H_{\pv}^{S} = H_{\pv}^{M}
$$
with equivalent norms.
\end{thm}
\begin{proof}
Let $f \in H_{\pv}^{S}$. By Lemma \ref{lemma}, there exists $h \in \cG_{\pv}$ and $g \in \cQ_{\pv}$ such that $f = h + g$ and $\|h\|_{\cG_{\pv}} \leq c \, \|f\|_{H_{\pv}^{S}}$ and $\|g\|_{\cQ_{\pv}} \leq c \, \|f\|_{H_{\pv}^{S}}$.
Then
\begin{eqnarray*}
\|f\|_{H_{\pv}^{M}} \leq \|h\|_{H_{\pv}^{M}} + \|g\|_{H_{\pv}^{M}} \leq \|h\|_{\cG_{\pv}} + c \, \|g\|_{\cQ_{\pv}} \leq c \, \|f\|_{H_{\pv}^{S}}.
\end{eqnarray*}
The reverse inequality
$$
\|f\|_{H_{\pv}^{S}} \leq c \, \|f\|_{H_{\pv}^M}
$$
can be proved similarly.
\end{proof}

For a martingale $f$, the \emph{martingale transform} is defined by
$$
(\cT f)_{n} := \sum_{k=1}^{\infty} b_{k-1} \, d_{k}f,
$$
where $b_k$ are $\cF_{k}$-measurable and $|b_{k}| \leq 1$. The martingale transform is bounded on $L_{\pv}$, if $1 < \pv < \infty$.

\begin{thm}\label{martingale transform korlatos}
If $1 < \pv < \infty$, then for all $f \in L_{\pv}$,
$$
\| \cT f \|_{\pv} \leq c \|f\|_{\pv}.
$$
\end{thm}
\begin{proof}
Because of $|b_k| \leq 1$, it is clear that $S(\cT f) \leq S(f)$. By Theorem \ref{BDG}, the spaces $H_{\pv}^M$ and $H_{\pv}^S$ are equivalent. Therefore using Theorem \ref{M korlatossaga},
\begin{eqnarray*}
\|\cT f\|_{\pv} \leq \|\cT f\|_{H_{\pv}^M} \leq c \|\cT f\|_{H_{\pv}^S} \leq c \|f\|_{H_{\pv}^S} \leq c \|f\|_{H_{\pv}^M} \leq c \|f\|_{\pv},
\end{eqnarray*}
which proves the theorem.
\end{proof}

\nocite{}



\end{document}